\documentclass{amsart}
\usepackage{amsthm}
\usepackage{amsmath}
\usepackage{amssymb}
\usepackage{mathrsfs}
\usepackage{enumerate}
\usepackage{float}
\usepackage{url, multicol}
\usepackage[all]{xy}
\usepackage{graphicx}
\usepackage{hyperref}

\newcommand{\Q}{{\mathbb Q}}
\newcommand{\C}{{\mathbb C}}
\newcommand{\Z}{{\mathbb Z}}
\newcommand{\F}{{\mathbb F}}
\newcommand{\N}{{\mathbb N}}
\newcommand{\R}{{\mathbb R}}

\newcommand{\OO}{{\mathcal O}}

\newcommand{\p}{{\mathfrak p}}

\renewcommand{\>}{\rangle}

\newcommand{\rank}{\mathrm{rank}}

\newcommand{\ord}{\mathrm{ord}}
\newcommand{\Gal}{\mathrm{Gal}}
\newcommand{\cyc}{\mathrm{cyc}}
\newcommand{\Pic}{\mathrm{Pic}}
\renewcommand{\ker}{\mathrm{ker}}
\newcommand{\im}{\mathrm{im}}

\newcommand{\Br}{\mathrm{Br}}

\newcommand{\Spec}{\mathrm{Spec}}
\newcommand{\Hom}{\mathrm{Hom}} 

\title[Generalization of Ferrero's Computations]{An Alternative Proof and Generalization of Ferrero's Computations of Iwasawa $\lambda$-Invariants}
\author{Jordan Schettler}

\begin{document}
\newtheorem{thm}{Theorem}
\newtheorem{conj}[thm]{Conjecture}
\newtheorem{prop}[thm]{Proposition}
\newtheorem{lemma}[thm]{Lemma}
\newtheorem{corollary}[thm]{Corollary}
\theoremstyle{remark}
\newtheorem{rem}[thm]{Remark}
\theoremstyle{definition}
\newtheorem{defn}[thm]{Definition}
\newtheorem{exam}[thm]{Example}
\numberwithin{equation}{section}
\numberwithin{equation}{thm}

\begin{abstract}
We prove a slight generalization of Iwasawa's `Riemann-Hurwitz' formula for number fields and use it to generalize Ferrero's and Kida's well-known computations of Iwasawa $\lambda$-invariants for the cyclotomic $\Z_2$-extensions of imaginary quadratic number fields. In particular, we show that if $p$ is a Fermat prime, then similar computations of Iwasawa $\lambda$-invariants hold for certain imaginary quadratic extensions of the unique subfield $k\subset \mathbb{Q}(\zeta_{p^2})$ such that $[k:\mathbb{Q}]=p$. In fact, we actually prove more by explicitly computing cohomology groups of principal ideals. The computation of lambda invariants obtained is a special case of a much more general result concerning relative lambda invariants for cyclotomic $\Z_2$-extensions of CM number fields due to Y\^{u}ji Kida. However, the approach used here significantly differs from that of Kida, and the intermediate computations of cohomology groups found here do not hold in Kida's more general setting.
\end{abstract}

\maketitle

%
%

\section{Introduction}

Fix a prime $p$ and number field $k$. Suppose $k_{\infty}$ is a $\Z_p$-extension of $k$, i.e., $k_{\infty}$ is a Galois extension of $k$ with Galois group $\Gal(k_{\infty}/k)$ isomorphic to the group $\Z_p$ of $p$-adic integers. The subfields of $k_{\infty}$ which contain $k$ all lie in a tower
\begin{align*}
k = k_0 \subset k_1 \subset k_2 \subset \ldots \subset k_{\infty}
\end{align*}
with $\Gal(k_n/k) \cong \Z/(p^n)$ for all integers $n \geq 0$. Kenkichi Iwasawa's well-known growth formula (see \cite{Iwas5} or \cite{Iwas4}) states that there are integers $\lambda, \mu, \nu$ such that if $A_n$ denotes the $p$-primary part of the class group of $k_n$, then
\begin{align*}
|A_n| = p^{\lambda \cdot n+\mu \cdot p^n + \nu}
\end{align*}
for all sufficiently large integers $n$. In particular, $\lambda,\mu \geq 0$ but we can have $\nu < 0$. We call $\lambda$, $\mu$, $\nu$ the Iwasawa invariants of the extension $k_{\infty}/k$. There is a special case in which all the invariants are known to vanish.
\begin{thm}[Iwasawa, \cite{Iwas5}]\label{vanish}
Let $k$ be a number field having exactly one prime $\p$ lying over a rational prime $p$, and let $k_{\infty}$ be a $\Z_p$-extension of $k$. Suppose $p$ does not divide the class number of $k$. Then for every number field $k_n \subset k_{\infty}$ which contains $k$, the prime $\p$ ramifies totally in $k_n/k$ and $p$ does not divide the class number of $k_n$; in particular, all of the Iwasawa invariants for $k_{\infty}/k$ are zero, i.e., $\lambda = \mu = \nu = 0$.
\end{thm}
For every prime $p$ and number field $k$, there is at least one $\Z_p$-extension of $k$; namely, there is a unique $\Z_p$-extension of $k$ contained in  $\bigcup_{n \geq 0} k(\zeta_{p^n})$ where the $\zeta_{p^n}$ are primitive $p^n$th roots of unity. We denote this $\Z_p$-extension by $k_{\infty}^{\cyc}$ and call it the cyclotomic $\Z_p$-extension of $k$. We write $\lambda_p(k)$, $\mu_p(k)$, $\nu_p(k)$ for the Iwasawa invariants of the extension $k^{\cyc}_{\infty}/k$. Conjecturally, the only $\Z_p$-extension of a totally real number field is the cyclotomic one, and this is known for $\Q$ and real quadratic number fields, for example.

Iwasawa conjectured that $\mu_p(k)=0$ for all primes $p$ and number fields $k$, and no counterexamples are known.\footnote{However, there are non-cyclotomic $\Z_p$-extensions of number fields which have an Iwasawa invariant $\mu > 0$.} This conjecture has been verified for abelian number fields by Bruce Ferrero and Lawrence Washington (\cite{Ferr2}) and for $p$-extensions of number fields $k$ with $\mu_p(k)=0$ by Iwasawa (\cite{Iwas6}). We make the assumption $\mu_p(k)=0$ throughout the paper.

Bruce Ferrero and Y\^{u}ji Kida independently calculated $\lambda_2(k)$ for imaginary quadratic fields $k$. Their computations are explicit:
\begin{thm}[Ferrero, \cite{Ferr}; Kida, \cite{Kida2}]\label{Fer}
Let $d > 2$ be a squarefree integer. Then
\begin{align}\label{kidaspecial}
\lambda_2(\Q(\sqrt{-d})) = -1 + \mathop{\sum_{p|d}}_{p\neq 2} 2^{\ord_2(p^2-1) - 3}
\end{align}
where the sum ranges over all odd primes $p$ dividing $d$.
\end{thm}
Here $\ord_2(n)$ denotes the 2-adic order of an integer $n$, i.e., the largest exponent $e$ such that $2^e|n$. Thus, for example, if $p$ is a prime with $p \equiv \pm 3 \pmod{8}$, then $\ord_2(p^2-1) = 3$, so $\lambda_2(\Q(\sqrt{-p})) = -1 + 2^{3-3} = 0$. In fact, there are infinitely many such primes, so for any integer $m \geq 0$ there are infinitely many choices of $m+1$ such distinct primes $p_1, \ldots, p_{m+1}$, and we get $\lambda_2(\Q(\sqrt{-p_1 p_2\cdots p_{m+1}})) = m$.

In \cite{KidaJ}, Kida further provided a less explicit computation of $\lambda^{-}_2(k)$ for any CM number field $k$. In more detail, let $k^+$ denote the maximal real subfield of $k$, and let $A^{\ast}(k^+_n)$ denote the $2$-primary part of the narrow class group of the $n$th level $k_n^+$ in the cyclotomic $\Z_2$-extension $k_{\infty}^+$ of $k^+$. Kida showed under mild assumptions\footnote{We need only assume that $\mu_2^{\ast}(k^+) = 0$ where $\mu_2^{\ast}(k^+)$ is the ``narrow'' mu-invariant , i.e., the mu-invariant in the Iwasawa's growth formula for the narrow class numbers $h^{\ast}_n$ of $k^+_n$.} that for sufficiently large $n$ we have
\begin{align}\label{kidageneral}
\lambda^{-}_2(k) \mathrel{\mathop :}= \lambda_2(k) - \lambda_2(k^+) = \delta - \tau - 1 + \dim_{\F_2} A^{\ast}(k^+_n)/A^{\ast}(k^+_n)^2 + s_n(k/k^+)
\end{align}
where $\delta$ is 1 or 0 depending on whether or not $k_{\infty}^{\cyc}$ contains a primitive 4th root of unity, $\tau$ is 1 or 0 depending on whether or not the ramification indices of the primes dividing $2$ in $k^+_{\infty}/\Q_{\infty}$ are all even, and $s_n(k/k^+)$ denotes the number finite primes of $k^{\cyc}_n$ which ramify in $k_n^{\cyc}/k^+_n$ and do not divide $2$.

We will recover Eq. \ref{kidageneral} for certain CM number fields $k$ having $\delta=\tau=\dim_{\F_2} A^{\ast}(k^+_n)/A^{\ast}(k^+_n)^2=0$ for all $n$. To do so, we will make use of a general Hurwitz formula for number fields. Such a formula was first proven by Kida in \cite{Kida} for extensions of CM number fields and then more generally by Iwasawa in \cite{Iwas} for extensions of number fields in which no infinite primes ramify. Both formulas need $\mu=0$ assumptions. We will use a generalization of Iwasawa's formula which does not need this assumption on the ramification of infinite places. In particular, we will be able to apply this general formula directly to an extension $\ell/k$ of a CM number field $\ell$ over its maximal real subfield $k$. Kida does not use the method in \cite{KidaJ} to derive Eq. \ref{kidageneral}.

These Hurwitz formulas for number fields mentioned above have their genesis in an idea coming from Iwasawa in \cite{Iwas3} where he argued that when $\mu_p(k)=0$, the invariant $\lambda_p(k)$ is a good analog for twice the genus of a curve.

\section{Iwasawa's `Riemann-Hurwitz' Formula Revisited}
Following Iwasawa, we  define a $\Z_p$-field to be the cyclotomic $\Z_p$-extension field of a number field. In other words, $K$ is a $\Z_p$-field when $K=k_{\infty}^{\cyc}$ for the prime $p$. Note that if $k_{\infty}^{\cyc} = \ell_{\infty}^{\cyc}$ for some prime $p$ and number fields $k$, $\ell$, then $\lambda_p(k) = \lambda_p(\ell)$ and we have $\mu_p(k) = 0 \Leftrightarrow \mu_p(\ell) = 0$. Thus for a $\Z_p$-field $K$ we may define $\lambda_K$ to the be Iwasawa $\lambda$-invariant $\lambda_p(k)$ for any number field $k$ with $K = k_{\infty}^{\cyc}$. Likewise, we may define the notation $\mu_K=0$ to indicate that $\mu_p(k) = 0$ for some (and hence every) number field $k$ such that $K = k_{\infty}^{\cyc}$.

\begin{thm}\label{first}
Suppose $K$ is a $\Z_p$-field. Let $I_K$ denote the group of invertible fractional ideals in the integer ring $\OO_K$ of $K$, and let $P_K \leq I_K$ denote the subgroup of principal fractional ideals. Then the $p$-primary part $A_K$ of the class group $C_K= I_K/P_K$ satisfies
\begin{align*}
A_K \cong (\Q_p/\Z_p)^{\lambda_K} \oplus M
\end{align*}
where $p^nM = 0$ for some $n \geq 0$. Moreover, $M$ is trivial precisely when $\mu_K=0$.
\end{thm}
This structure theorem, whose proof may be gleaned from \cite{Iwas5} and \cite{Iwas4}, allows us to pin down the analogy between number fields and curves in the spirit of \cite{Iwas3}. The analogy is illustrated in Table \ref{blahintro}.
\begin{table}
\begin{center}
\begin{tabular}{|p{5.9cm}|p{5.9cm}|}
\hline
If $K$ is the function field of a smooth, projective curve $X$ over $\mathbb{C}$, then the local rings of $X$ at closed points are DVRs, and the $p$-primary part of the Picard group satisfies

~

\hspace{.5 in} $\Pic(X)[p^{\infty}] \cong (\Q_p/\Z_p)^{2g_K}$

~

where $g_K$ is the genus.

&

If $K$ is a $\Z_p$-field with $\mu_K=0$ and $X = \Spec(\OO_K[1/p])$, then the local rings of $X$ at closed points are DVRs, and the $p$-primary part of the Picard group satisfies

~

\hspace{0.5 in} $\Pic(X)[p^{\infty}] \cong (\Q_p/\Z_p)^{\lambda_K}$

~

where $\lambda_K$ is an Iwasawa invariant.

\\
\hline
\end{tabular}
\end{center}
\caption{Similarities in structures of Picard groups}
\label{blahintro}
\end{table}
A few remarks are in order. The ring $\OO_K$ is not Noetherian, but by inverting $p$, the resulting ring $\OO_K[1/p]$ is actually a Dedekind domain, so its prime spectrum is a nice scheme which shares many properties with curves. This also helps to explain why the ramification for the prime $p$ is missing in Iwasawa's `Riemann-Hurwitz' formula; however, we will not use this interpretation in the proof of Iwasawa's formula, but the so-called Dedekind different formula implies that a purely geometric proof of Iwasawa's result should exist. Before stating Iwasawa's `Riemann-Hurwitz' formula, we need a definition.
\begin{defn}
Let $G$ be a cyclic $p$-group. For a $G$-module $M$, define the `Euler characteristic' $\chi(G,M) \in \Z$ to be the exponent of $p$ in the Herbrand quotient
\begin{align*}
p^{\chi(G,M)} = \frac{|H^2(G,M)|}{|H^1(G,M)|}
\end{align*}
when both cohomology groups $H^i(G,M)$ are finite for $i=1,2$. Note that $\chi$ inherits the following properties from the Herbrand quotient:
\begin{enumerate}
\item $\chi$ is additive on short exact sequences of $G$-modules\footnote{If $0\rightarrow A \rightarrow B \rightarrow C \rightarrow 0$ is a SES of $G$-modules and two of the numbers $\chi(G, A)$, $\chi(G, B)$, $\chi(G, C)$ are finite (i.e., well-defined), then so is the other and $\chi(G, A) -\chi(G, B) + \chi(G, C)= 0$.}
\item $\chi(G,M) = 0$ when $M$ is a finite $G$-module
\item $\chi(G, M^{\ast}) = -\chi(G,M)$ when $M^{\ast} = \Hom_{\Z_p}(M, \Q_p/\Z_p)$ is the $p$-Pontryagin dual of a $\Z_pG$-module $M$.
\end{enumerate}
\end{defn}
\begin{thm}[Iwasawa's `Riemann-Hurwitz' Formula]\label{Iw}
Let $L/K$ be a $\Z/(p)$-extension of $\Z_p$-fields with $G = \Gal(L/K)$. Suppose $\mu_K=0$. Then $\mu_L=0$ and
\begin{equation}\label{equat}
\lambda_L = p\lambda_K - (p-1)\chi(G, P_L) + \sum_{w \nmid p} (e(w)-1)
\end{equation}
where the sum ranges over all finite places $w$ of $L$ not lying above $p$, $e(w)$ is the ramification index of $w$ in $L/K$, and $G$ acts in the obvious way on the principal fractional ideals $P_L$ of the integer ring $\OO_L$ in $L$.
\end{thm}
\begin{rem}
If the extension $L/K$ in Theorem \ref{Iw} is unramified at the infinite places (always true, e.g., when $p$ is odd), then Iwasawa showed that $H^2(G,L^{\times}) = 0$, so additivity of the `Euler characteristic' $\chi$ shows that $-\chi(G,P_L) = \chi(G,\OO_L^{\times})$, which recovers the formula as originally stated by Iwasawa in \cite{Iwas}.
\end{rem}
\begin{proof}[Proof of Theorem \ref{Iw}]
Using the notation of Theorem \ref{first} above, we let $A_L$ denote the $p$-primary part of the class group $C_L = I_L/P_L$ where again $I_L, P_L$ are the groups of invertible and principal, respectively, fractional ideals of the integer ring $\OO_L$. The statement that $\mu_L=0$ follows from $\mu_K = 0$ by a result mentioned above (see \cite{Iwas6}) since $L/K$ is a $p$-extension, so the structure theorem implies
\begin{align*}
A_L^{\ast} = \Hom_{\Z_p}(A_L, \Q_p/\Z_p) \cong \Hom_{\Z_p}((\Q_p/\Z_p)^{\lambda_L}, \Q_p/\Z_p) \cong \Z_p^{\lambda_L}
\end{align*}
as $\Z_p$-modules. On the other hand, we have the following classification theorem.
\begin{thm}[Diederichsen, \cite{Died}]\label{p}
Let $\<g\> = G \cong \Z/(p)$. The only indecomposable $\Z_pG$-modules which are free of finite rank over $\Z_p$ are (up to isomorphism) $\Z_p$, $\Z_pG$, and the augmentation ideal $I_pG = (g-1)\Z_pG$.
\end{thm}
Hence
\begin{align*}A_L^{\ast} \cong \Z_p^a \oplus (\Z_pG)^b \oplus (I_pG)^c\end{align*}
as $\Z_pG$-modules for some nonnegative integers $a,b,c$. It is easy to compute the $\Z_p$-ranks, $G$-invariants, and Euler characteristics, of these indecomposables. The results are summarized in Table \ref{tab0}.
\begin{table}
\begin{align*}
\begin{array}{c|c|c|c|c|c|}
 & \mbox{rank}_{\Z_p}(-) & (-)^G & H^2(G, -) & H^1(G, -) & \chi(G, -) \\ \hline
 \Z_p & 1 & \Z_p & \Z_p/p\Z_p & 0 & 1 \\ \hline
 \Z_pG & p & \Z_p & 0 & 0 & 0 \\ \hline
 I_pG & p-1 & 0 & 0 & \Z_p/p\Z_p & -1 \\ \hline
\end{array}
\end{align*}
\caption{}
\label{tab0}
\end{table}
Moreover, if $S$ is the set of finite places of $K$ which do not lie above $p$ and ramify in $L/K$, then Iwasawa showed in \cite{Iwas} that
\begin{align*}
\chi(G,I_L) = |S|.
\end{align*}
Also, the last column in Table \ref{tab0} implies
\begin{align*}
\chi(G,A_L^{\ast}) = a \cdot 1 + b \cdot 0 + c\cdot(-1) = a-c.
\end{align*}
Hence duality gives
\begin{align*}
\chi(G,A_L) = -\chi(G,A_L^{\ast}) = -a+c,
\end{align*}
but
\begin{align*}
\chi(G,C_L) = \chi(G,A_L)
\end{align*}
since $G$ is a $p$-group and $C_L$ is torsion, so $\chi(G, P_L)$ is also finite and additivity gives
\begin{align*}
-\chi(G, P_L) + |S| = - \chi(G,P_L) +  \chi(G,I_L) = \chi(G, C_L) = -a+c.
\end{align*}
Also, the natural map
\begin{align*}
A_K \rightarrow A_L^G
\end{align*}
has finite kernel and finite cokernel since the same is true of $C_K \rightarrow C_L^G$ as may be seen from the snake lemma. Thus the first two columns in Table \ref{tab0} show that
\begin{align*}\lambda_K &= \rank_{\Z_p}(A_K^{\ast}) = \rank_{\Z_p}((A_L^G)^{\ast}) =\rank_{\Z_p}((A_L^{\ast})^G) \\
&= a \cdot 1 + b \cdot 1 + c \cdot 0 = a+b.\end{align*}
Putting all of this together, we get
\begin{align*}
\lambda_L &= a \cdot 1 + b \cdot p + c (p-1) = p(a+b) + (p-1)(-a+c) \\
&= p\lambda_K - (p-1)\chi(G, P_L) + (p-1)|S|,
\end{align*}
as needed. We have shown more than just a formula for the $\lambda$-invariants; in fact,
\begin{align*}
A_L^{\ast} \cong \Z_p^a \oplus (\Z_pG)^{\lambda_K - a} \oplus (I_pG)^{|S|-\chi(G,P_L)+a}
\end{align*}
as $\Z_pG$-modules for some nonnegative integer $a$ with $\chi(G,P_L) - |S| \leq a \leq \lambda_K$.
\end{proof}
\section{Main Result}
\begin{thm}\label{main}
Suppose $p$ is a prime of the form $2^t+1$ for some integer $t \geq 0$ and let $d > 2 \geq (d, p)$ be a squarefree integer. Take $k$ to be the unique real subfield of $\Q(\zeta_{2p^2})$ such that $[k:\Q]=p$, and take $K = k_{\infty}^{\cyc}$ to be the cyclotomic $\Z_2$-extension of $k$. If the class number $h_k$ of $k$ is odd (e.g., we can take $p \in \{2, 3, 5, 17, 257\}$), then 
\begin{align*}
|H^1(G, P_L)|= 1 \mbox{ and } |H^2(G,P_L)|=2
\end{align*}
where $L = K(\sqrt{-d})$, $G = \Gal(L/K)$, and $P_L$ denotes the principal fractional ideals of $L$. In particular, we have $\chi(G,P_L)=1$, and Theorem \ref{Iw} implies the following special case of \cite{KidaJ}:
\begin{align*}\lambda_2(k(\sqrt{-d})) = -1 + |S| \end{align*}
where $S$ is the set of finite places of $K$ not lying above $2$ which ramify in $L/K$.
\end{thm}
\begin{rem}
Note that the field $k$ in Theorem \ref{main} is the first layer in the cyclotomic $\Z_p$-extension of $\Q$. Thus when $p=2$, Theorem \ref{main} precisely recovers Theorem \ref{Fer} since, in that case, $K = k_{\infty}^{\cyc} = \Q_{\infty}$ is the cyclotomic $\Z_2$-extension of $\Q$ and for each odd prime $q$ there are exactly $2^{\ord_2(q^2-1)-3}$ primes of $\Q_{\infty}^{\cyc}$ which lie above $q$.
\end{rem}
\begin{proof}[Proof of Theorem \ref{main}]
We first apply the general form of Iwasawa's `Riemann-Hurwtiz' formula (Equation \ref{equat}) to the extension $L/K$ of $\Z_2$-fields where we take $L=K(\sqrt{-d}) =\ell_{\infty}^{\cyc}$ to be the cyclotomic $\Z_2$-extension of $\ell = k(\sqrt{-d})$ and as above $G = \Gal(L/K)$. We get
\begin{align*}
\lambda_2(k(\sqrt{-d})) = \lambda_L = 2\lambda_K - \chi(G,P_L) + |S|.
\end{align*}
Thus it suffices to show that $\lambda_K = 0$, $|H^1(G, P_L)|= 1$, and $|H^2(G,P_L)|=2$.

First, we prove that $\lambda_K = 0$ using Theorem \ref{vanish}. We have assumed that $h_k$ is odd, and, in fact, this assumption is valid\footnote{This follows, e.g., from results of Humio Ichimura and Shoichi Nakajima; see Proposition 1 and its proof in Section 3 of \cite{Ichi}} for $p \in \{2, 3, 5, 17, 257\}$. Hence it is enough to show that $k$ has exactly one prime lying above $2$. When $p=2$, this is clear since $k = \Q(\sqrt{2})$, so we may assume $p = 2^t+1$ for some integer $t \geq 1$. Of course, we must have $t = 2^r$ (i.e, $p = F_r = 2^{2^r}+1$ is a prime Fermat number) for some integer $r \geq 0$. For every $j \geq 1$, we can factor
\begin{align*}
F_j - 2 = 2^{2^j}-1 = F_0F_1\cdots F_{j-1}
\end{align*}
as a product of consecutive Fermat numbers $F_i = 2^{2^i}+1$.  This identity shows that Fermat numbers are pairwise relatively prime, so $p^2 \nmid 2^{2^j}-1$ for any $j \geq 0$ since $p=F_r$ is a prime Fermat number. This means that the multiplicative order of $2$ modulo $p^2$ is not a power of $2$ which forces the residue degree of $2$ in $\Q(\zeta_{p^2}) = \Q(\zeta_{2p^2})$ to be divisible by $p$. Consequently, the residue degree of $2$ in $k$ is $p$, which is equivalent to $2$ being inert in $k$. Hence $\lambda_K=0$. 

It remains to show that $|H^1(G, P_L)|= 1$ and $|H^2(G,P_L)|=2$. First, we fix some notation. Write out the towers for the $\Z_2$-extensions $K/k$ and $L/\ell = K(\sqrt{-d})/k(\sqrt{-d})$ as
\begin{align*}
k &\subset k_1 \subset k_2 \subset \ldots \subset k_{\infty} := K \\
\ell &\subset \ell_1 \subset \ell_2 \subset \ldots \subset \ell_{\infty} := L.
\end{align*}
Note that we have dropped the superscripted cyclotomic ``$\cyc$'' notation for convenience. In this way,
\begin{align*}\Gal(k_n / k) \cong \Gal(\ell_n/\ell) \cong \Z/(2^n) \mbox{ for all positive integers } n \in \N
\end{align*}
and
\begin{align*}
G_n := \Gal(\ell_n/k_n) \cong \Z/(2) \mbox{ for all indices } n \in \N \cup \{\infty\}.
\end{align*}
Thus we have exact sequences
\begin{align*} 0 &\rightarrow H^1(G_n, P_{\ell_n}) \rightarrow H^2(G_n, \OO_{\ell_n}^{\times}) \rightarrow H^2(G_n,\ell_n^{\times}) \\
&\rightarrow H^2(G_n, P_{\ell_n}) \rightarrow H^1(G_n, \OO_{\ell_n}^{\times}) \rightarrow 0, \end{align*}
and we have norm maps $N_n \colon \ell_n \rightarrow k_n$ for all indices $n \in \N \cup \{\infty\}$. In this case, $N_n(\alpha) = \alpha\bar{\alpha} =|\alpha|^2$ is just the square modulus since the restriction of complex conjugation generates $G_n$ for all indices $n \in \N \cup \{\infty\}$. Thus the images of these norms maps consist entirely of totally positive\footnote{Recall that an algebraic number $\alpha$ is called totally positive if the images of $\alpha$ under every embedding $\Q(\alpha) \hookrightarrow \mathbb{C}$ are real and positive.} elements.

Now we need a theorem. It generalizes a result of Weber's (see S\"{a}tze 6 and 25 in \cite{Hass}) which Ferrero used in \cite{Ferr}.
\begin{thm}[I. Hughes and R. Mollin, \cite{Hugh}]\label{hughes}
Let $F'/F$ be a cyclic $2$-extension of real abelian number fields. Suppose $\Gal(F/\Q)$ has exponent $n$ such that $-1$ is congruent to a power of $2$ modulo $n$, and, if $F \neq F'$, suppose that exactly one prime ramifies in $F'/F$. If the class number $h_F$ of $F$ is odd, then every totally positive element of $\OO_{F'}^{\times}$ is a square in $\OO_{F'}^{\times}$.
\end{thm}
For all $n \in \N$ we apply Theorem \ref{hughes} to the extension $F'/F = k_n/k$ to get
\begin{align}\label{n}\OO_{k_n}^{\times} \cap N_{\infty}(L^{\times}) = (\OO_{k_n}^{\times})^2 = N_n(\OO_{\ell_n}^{\times}) \end{align}
where $(\OO_{k_n}^{\times})^2$ denotes the subgroup of squares of units. In fact, taking unions shows that Equation \ref{n} also holds for $n = \infty$, so
\begin{align*}H^1(G,P_L) \cong \ker(\OO_K^{\times}/(\OO_K^{\times})^2 \rightarrow K^{\times}/N_{\infty}(L^{\times})) = (\OO_K^{\times} \cap N_{\infty}(L^{\times}))/(\OO_K^{\times})^2 = 0,  \end{align*}
and likewise $H^1(G_n, P_{\ell_n}) = 0$ for all $n\in \N$.

Hence it remains only to show that $|H^2(G,P_L)|=2$. To do this, we first prove $|H^1(G,\OO_L^{\times})| = 2$ and then show that the surjection $H^2(G,P_L) \rightarrow H^1(G,\OO_L^{\times})$ is also an injection.

For each $n\in \N$, let $t_{\infty}(n)$ denote the number of infinite places of $k_n$ which ramify in $\ell_n/k_n$. Since $k_n$ is totally real and $\ell_n$ is totally complex, $t_{\infty}(n) = [k_n:\Q]$ is just the number of real places of $k_n$, so Dirichlet's unit theorem gives
\begin{align*}
\OO_{k_n}^{\times}\cong \Z^{t_{\infty}(n)+0-1} \oplus \frac{\Z}{(2)}
\end{align*}
as abelian groups. Then using Theorem \ref{hughes} again shows
\begin{align}\label{unit}
|H^2(G_n, \OO_{\ell_n}^{\times})| = \left| \frac{\OO_{k_n}^{\times}}{(\OO_{k_n}^{\times})^2} \right| = \left| \frac{ \Z^{t_{\infty}(n)-1} \oplus (\Z/(2))}{2(\Z^{t_{\infty}(n)-1} \oplus (\Z/(2)))} \right| = 2^{t_{\infty}(n)-1+1} = 2^{t_{\infty}(n)}.
\end{align}
Now we state another needed result, whose proof may be found, for example, in \cite{Gree}.
\begin{thm}\label{greenberg}
Suppose $F'/F$ is a quadratic extension of number fields and let $t_{\infty}$ denote the number of infinite places of $F$ which ramify in $F'$. Then
\begin{align*}\chi(\Gal(F'/F), \OO_{F'}^{\times}) = t_{\infty} - 1.\end{align*}
\end{thm}
We now apply Theorem \ref{greenberg} for the extension $F'/F = \ell_n/k_n$ to conclude that
\begin{align*}
|H^1(G_n, \OO_{\ell_n}^{\times})| = 2
\end{align*}
for all positive integers $n \in \N$ since we have already shown in Equation \ref{unit} above that $|H^2(G_n, \OO_{\ell_n}^{\times})| = 2^{t_{\infty}(n)}$. On the other hand, for $n \in \N \cup \{\infty\}$ we have
\begin{align*} U_n/V_n \cong  H^1(G_n, \OO_{\ell_n}^{\times}) \end{align*}
where $U_n$ is the norm 1 units in $\OO_{\ell_n}$ and $V_n = \{\bar{u}/u: u \in \OO_{\ell_n}^{\times}\}$. We claim $U_n/V_n$ is generated by the coset of $-1$ for all indices $n \in \N \cup \{\infty\}$. We need the following lemma which we do not prove here but is not hard to establish (see, for example, Lemma 6.15 in \cite{Sche}).
\begin{lemma}
Let $d > 1$ be a squarefree integer. Suppose $F$ is a number field with discriminant $\Delta_F$ such that $(d, \Delta_F)\leq 2$. Then $4\OO_{F(\sqrt{-d})} \subseteq \OO_F + \sqrt{-d}\OO_F$.
\end{lemma}
Suppose contrary to the claim that $-1 = \bar{u}/u \in V_n$ for some $u \in \OO_{\ell_n}^{\times}$ and some positive integer $n \in \N$. Then both $u$ and $u^{-1}$ are in $\OO_{\ell_n}$, so since $u = -\bar{u}$ and $u^{-1} = - \overline{u^{-1}}$ the lemma implies
\begin{align*}u = \frac{a\sqrt{-d}}{4} \mbox{ and } u^{-1} = \frac{b\sqrt{-d}}{4} \end{align*}
for some $a, b \in \OO_{k_n}$. Hence $abd = - 4^2$, so $d$ divides $4^2=2^4$ in $\OO_{k_n}$, but that means $d$ divides $2^{4}$ in $\Z$. Therefore $d=1$ or $d=2$ since $d$ is a squarefree positive integer, which contradicts our assumption that $d>2$. Thus for each $n \in \N$ we know that $-1 \notin V_n$ and $|U_n/V_n| = 2$, so $U_n/V_n$ is generated by the coset of $-1$. It follows that $H^1(G, \OO_L^{\times}) \cong U_{\infty}/V_{\infty}$ is also generated by the coset of $-1$ and has order 2 since
\begin{align*}U_{\infty} = \bigcup_{n\in \N} U_n = \bigcup_{n\in \N} \langle-1\rangle V_n = \langle -1 \rangle V_{\infty} \mbox{ and }-1 \notin \bigcup_{n \in \N} V_n = V_{\infty}.\end{align*}
To summarize, we have an exact sequence
\begin{align*}
0 \rightarrow H^2(G,\OO_L^{\times}) \rightarrow H^2(G,L^{\times}) \rightarrow H^2(G, P_L) \rightarrow H^1(G, \OO_L^{\times}) \rightarrow 0
\end{align*}
where $|H^1(G, \OO_L^{\times})| = 2$, so to prove $|H^2(G,P_L)|=2$ (and thus finish the proof of Theorem \ref{main}) we only need to show that the map $H^2(G,\OO_L^{\times}) \rightarrow H^2(G,L^{\times})$ is onto since that would imply the map $H^2(G,L^{\times}) \rightarrow H^2(G, P_L)$ is trivial and, consequently, that the map $H^2(G, P_L) \rightarrow H^1(G, \OO_L^{\times})$ is bijective.

We have a commutative square
\begin{align*}
\xymatrix{
H^2(G, \OO_L^{\times}) \ar[r] \ar[d]^-{\mbox{\rotatebox{90}{$\sim$}}} & H^2(G, L^{\times}) \ar[d]^-{\mbox{\rotatebox{90}{$\sim$}}} \\
 \displaystyle{\varinjlim_n} \, \frac{\OO_{k_n}^{\times}}{N_n(\OO_{\ell_n}^{\times})} \ar[r] & \displaystyle{\varinjlim_n} \, \frac{k_n^{\times}}{N_n(\ell_n^{\times})}
}
\end{align*}
where the horizontal maps are the natural maps and the vertical maps are isomorphisms. Pick some $n \in \N$ and $x_n \in k_n^{\times}$. Then the commutative square above implies that proving the surjectivity of $H^2(G,\OO_L^{\times}) \rightarrow H^2(G,L^{\times})$ amounts to showing there is an $m \in \N$ and a $y_m \in \OO_{k_m}^{\times}$ such that we have an equality of cosets $x_nN_j(\ell_j^{\times})= y_mN_j(\ell_j^{\times})$ for some integer $j\geq \max\{n,m\}$. In fact, we show that we can take $j=m+1$.

For the moment, let $m$ be an arbitrary positive integer. Define $S_m$ to be the set of places of $k_m$ which do not split in $\ell_m$. We have a commutative diagram
\begin{align}\label{diagram}
\xymatrixcolsep{4pc}\xymatrix{
\displaystyle{\frac{\OO_{k_{m+1}}^{\times}}{N_{m+1}(\OO_{\ell_{m+1}}^{\times})}} \ar[r]^{\beta_{m+1}} & \displaystyle{\frac{k_{m+1}^{\times}}{N_{m+1}(\ell_{m+1}^{\times})}} \ar[r]^-{\sim} & \displaystyle{\bigoplus_{v'\in S_{m+1}}}\hspace{-0.1 in}' \, \Br(\ell_{m+1,w'}/k_{m+1,v'}) \\
\displaystyle{\frac{\OO_{k_m}^{\times}}{N_m(\OO_{\ell_m}^{\times})}} \ar[r]^{\beta_m} \ar[u]^{\alpha_m} & \displaystyle{\frac{k_m^{\times}}{N_m(\ell_m^{\times})}}  \ar[u]^{\gamma_m} \ar[r]^-{\sim} &  \displaystyle{\bigoplus_{v \in S_m}}\hspace{-0.03 in}'  \, \Br(\ell_{m,w}/k_{m,v}) \ar[u]^{\delta_m} 
}
\end{align}
where $\alpha, \beta, \delta$ are the natural maps and the restricted direct sum $\oplus'$ on relative local Brauer groups contains only those tuples which lie in the kernel of the natural map to $\Q/\Z$ given by the sum of local invariants.\footnote{Note that $v\in S_m$ uniquely specifies a place $w$ of $\ell_m$ which lies above $v$. Also, if $v\notin S_m$ is a place of $k_m$, then $\Br(\ell_{m,w}/k_{m,v})$ is trivial for any place $w$ of $\ell_m$ lying over $v$.} We have already noted that $\beta_m, \beta_{m+1}$ are injective, and $\alpha_m$ is injective for the same reason, i.e., Theorem \ref{hughes} implies \begin{align*}
\OO_{k_m}^{\times} \cap N_{m+1}(\OO_{\ell_{m+1}}) =( \OO_{k_m}^{\times})^2 = N_m(\OO_{\ell_m}^{\times}).
\end{align*}
Thus $\beta_{m+1} \circ \alpha_m$ is injective, so commutativity of the left side of Diagram \ref{diagram} and Equation \ref{unit} imply
\begin{align}\label{imply}
|\im(\gamma_m \circ \beta_m)| = |\im(\beta_{m+1} \circ \alpha_m)| = |\OO_{k_m}^{\times}/N_m(\OO_{\ell_m}^{\times})| = 2^{t_{\infty}(m)}.
\end{align}
Let $H_m$ be the subgroup consisting of those tuples in $\bigoplus_{v \in S_m}' \, \Br(\ell_{m,w}/k_{m,v})$ which are trivial on the components corresponding to finite places $v \in S_m$ which split in $k_{m+1}$.
We claim that the image of $H_m$ under $\delta_m$ also has size $2^{t_{\infty}(m)}$. First, note that for each $v\in S_m$ the relative local Brauer group $\Br(\ell_{m,w}/k_{m,v})$ has order two since
\begin{align*}
\Br(\ell_{m,w}/k_{m,v}) \cong \left\{ \begin{array}{ll} \ker(\Br(\R) \longrightarrow \Br(\C)=0) = \Br(\R) \cong \Z/(2) & \mbox{if $v$ is infinite} \\ \ker(\Q/\Z \stackrel{\times 2}{\longrightarrow} \Q/\Z) = \frac{1}{2}\Z/\Z \cong \Z/(2) & \mbox{if $v$ is finite.}\end{array} \right.
\end{align*}
If $v \in S_m$ is a finite place which does not split in $k_{m+1}$, then the map $\delta_m$ kills the component of $\Br(\ell_{m,w}/k_{m,v})$ since the local degree is $[k_{m+1}:k_m] = 2$. Also, the infinite places in $k_{\infty}/k$ are totally split, so if $v \in S_m$ is an infinite place, then there are exactly two (real) places $v_1', v_2'$ of $k_{m+1}$ which lie above $v$; in this case, the map $\delta_m$ on the component $\Br(\ell_{m,w}/k_{m,v})$ is given by
\begin{align*}
\Br(\ell_{m,w}/k_{m,v}) \rightarrow \Br(\ell_{m+1,w_1'}/k_{m+1,v_1'}) \oplus \Br(\ell_{m+1,w_2'}/k_{m+1,v_2'}) \colon \alpha \mapsto (\alpha, \alpha).
\end{align*}
Therefore $|\delta_m(H_m)| =2^{t_{\infty}(m)}$ as claimed since, in particular, $H_m$ contains all tuples whose only nonzero components correspond to infinite places and, if needed to make the sum of local invariants zero, exactly one finite place that does not split in $k_{m+1}$.

Now we specify $m$. Finite places are finitely split in the cyclotomic $\Z_p$-extension of a number field, so there is a sufficiently large positive integer $m\geq n$ such that every finite place $v\in S_m$ with $\ord_v(x_n)\neq 0$ does not split in $k_{m+1}$. Increasing $m$ if necessary, we may also assume that the finite places $v\in S_m$ which ramify in $\ell_m/k_m$ do not split in $k_{m+1}$. For finite places $v\in S_m$ which are unramified in $\ell_m/k_m$ with $\ord_v(x_n) = 0$, we have $x_n \in \OO_{k_{m,v}}^{\times}$, but $N_m(\OO_{\ell_{m,w}}^{\times}) = \OO_{k_{m,v}}^{\times}$ (see Theorem 2 in Chapter 31 of \cite{Lore}), so the image of $x_nN_m(\ell_m^{\times})$ in $\bigoplus_{v \in S_m}' \, \Br(\ell_{m,w}/k_{m,v}) \cong \bigoplus_{v \in S_m}' \, k_{m,v}^{\times}/N_m(\ell_{m,w}^{\times})$ is contained in $H_m$. Equivalently,
\begin{align}\label{in}x_nN_m(\ell_m^{\times}) \in \tilde{H}_m\end{align}
where $\tilde{H}_m$ denotes the subgroup of $k_m^{\times}/N_m(\ell_m^{\times})$ which maps isomorphically onto $H_m$ in the above commutative diagram \ref{diagram}. We have $|\gamma_m(\tilde{H}_m)| = |\delta_m(H_m)| =2^{t_{\infty}(m)}$ and we will show that $\im(\gamma_m \circ \beta_m) \subseteq \gamma_m(\tilde{H}_m)$, so Equation \ref{imply} implies via the pigeonhole principle that
\begin{align}\label{last}
\gamma_m(\tilde{H}_m)=\im(\gamma_m \circ \beta_m).
\end{align}
To see $\im(\gamma_m \circ \beta_m) \subseteq \gamma_m(\tilde{H}_m)$, we first note that the study of Brauer groups for local fields shows $\im(\beta_m)$ maps injectively into the subgroup of $\bigoplus_{v \in S_m}' \, \Br(\ell_{m,w}/k_{m,v})$ consisting of tuples which are trivial on all components other than those corresponding to infinite places and to finite places which ramify in $\ell_m/k_m$. Since the finite places $v\in S_m$ which ramify in $\ell_m/k_m$ are non-split in $k_{m+1}$ we know that $\im(\beta_m) \subseteq \tilde{H}_m$, and hence $\im(\gamma_m \circ \beta_m) \subseteq \gamma_m(\tilde{H}_m)$, as needed. Finally, Equations \ref{in} and \ref{last} give
\begin{align*}
 x_n N_{m+1}(\ell_{m+1}^{\times})=\gamma_m(x_nN_m(\ell_m^{\times})) = \gamma_m(\beta_m(y_mN_m(\OO_{\ell_m}^{\times}))) = y_mN_{m+1}(\ell_{m+1}^{\times})
\end{align*}
for some $y_m \in \OO_{k_m}^{\times}$, which finishes the proof.
\end{proof}

\bibliographystyle{amsalpha}
\bibliography{../References}

\vfill

\end{document}